\documentclass[leqno,12pt]{article} 
\usepackage{amsmath, amssymb}
\usepackage{amsthm} 
\newcommand\supp{\mathop{\rm supp}}

\theoremstyle{plain} 
\newtheorem{theorem}{\indent\sc Theorem}[section]
\newtheorem{lemma}[theorem]{\indent\sc Lemma}

\newtheorem{proposition}[theorem]{\indent\sc Proposition}

\theoremstyle{definition} 
\newtheorem{definition}[theorem]{\indent\sc Definition}
\newtheorem{remark}[theorem]{\indent\sc Remark}

\title{The $H^p(\mathbb{Z}^n) - H^q(\mathbb{Z}^n)$ boundedness of \\ the discrete Riesz potential} 
\author{
\textsc{Pablo Rocha} 
}
\date{} 
\begin{document}

\maketitle


\begin{abstract}
In [J. Class. Anal., vol. 26 (1) (2025), 63-76], we proved that the discrete Riesz potential $I_{\alpha}$ is a bounded operator 
$H^p(\mathbb{Z}^n) \to H^q(\mathbb{Z}^n)$ for $\frac{n-1}{n} < p \leq 1$, $\frac{1}{q} = \frac{1}{p} - \frac{\alpha}{n}$ and 
$0 < \alpha < n$. In this note, we extend such boundedness on the full range $0 < p \leq 1$. 
\end{abstract}

\section{Introduction}

Let $0 < \alpha < n$, the Riesz potential $\mathcal{R}_{\alpha}$ on $\mathbb{R}^n$ is defined by
\begin{equation} \label{Ralfa}
(\mathcal{R}_{\alpha} f)(x) = \int_{\mathbb{R}^n} f(y) |x-y|^{\alpha-n} dy.
\end{equation}
The well known Hardy-Littlewood-Sobolev inequality establishes that if $1 < p < \frac{n}{\alpha}$ and 
$\frac{1}{q} = \frac{1}{p} - \frac{\alpha}{n}$, then the operator $\mathcal{R}_{\alpha}$ is bounded from $L^{p}(\mathbb{R}^n)$ into 
$L^{q}(\mathbb{R}^n)$. For $p=1$ and $q = n/(n-\alpha)$, $\mathcal{R}_{\alpha}$ is of weak-type $(1, q)$ (see \cite{Stein}). For the 
endpoint $p = \frac{n}{\alpha}$ is known that $\mathcal{R}_{\alpha}$ is not a bounded operator 
$L^{\frac{n}{\alpha}}(\mathbb{R}^n) \to L^{\infty}(\mathbb{R}^n)$, however it can see that 
$\mathcal{R}_{\alpha} : L^{\frac{n}{\alpha}}(\mathbb{R}^n) \to BMO(\mathbb{R}^n)$ is bounded. This last estimate can also be obtained 
by duality. Indeed, $\mathcal{R}_{\alpha} : H^1(\mathbb{R}^n) \to L^{\frac{n}{n-\alpha}}(\mathbb{R}^n)$ is bounded (see \cite{Weiss}), 
where $H^1(\mathbb{R}^n)$ is the $1$-Hardy space on $\mathbb{R}^n$. In \cite{Feff}, C. Fefferman and E. Stein proved that $(H^1)^{*} = BMO$, 
this remarkable result joint with the fact that $(L^{\frac{n}{n-\alpha}})^{*} = L^{\frac{n}{\alpha}}$ allow to obtain, by duality, 
the $L^{\frac{n}{\alpha}}(\mathbb{R}^n) \to BMO(\mathbb{R}^n)$ boundedness of $\mathcal{R}_{\alpha}$.

The real Hardy spaces $H^p(\mathbb{R}^n)$, $\frac{n-1}{n} < p \leq \infty$, were first introduced by E. Stein and G. Weiss in \cite{Weiss}. Therein, the authors describe the $H^p$ theory in terms of systems of conjugate harmonic functions, for it the theory was only developed on the range $\frac{n-1}{n} < p \leq 1$ (see \cite[\S 5.16]{St}). They also showed that $\mathcal{R}_{\alpha}$ is a bounded operator from 
$H^p(\mathbb{R}^n)$ into $H^q(\mathbb{R}^n)$, when $\frac{n-1}{n} < p < \frac{n}{\alpha}$ and $\frac{1}{q} = \frac{1}{p} - \frac{\alpha}{n}$. 
Afterward, C. Fefferman and E. Stein in \cite{Feff} introduced real variable methods to characterize the Hardy spaces by means of maximal functions on the full range $0 < p \leq \infty$. This second approach brought greater flexibility to the theory. For $1 < p \leq \infty$, 
one has that $H^{p}(\mathbb{R}^{n}) \cong L^{p}(\mathbb{R}^{n})$, $H^{1}(\mathbb{R}^{n}) \subset L^{1}(\mathbb{R}^{n})$ strictly, and for 
$0 < p < 1$ the spaces $H^{p}(\mathbb{R}^{n})$ and $L^{p}(\mathbb{R}^{n})$ are not comparable.

The spaces $H^{p}(\mathbb{R}^{n})$ can also be characterized by atomic decompositions. That is, every distribution $f \in H^{p}$ can be expressed by
\[
f = \sum_j \lambda_j a_j, 
\]
where the $a_j$'s are $p$ - atoms, $\{ \lambda_j \} \in \ell^{p}$ and $\| f \|^{p}_{H^{p}} \approx \sum_j | \lambda_j |^{p}$.
For $0 < p \leq 1$, an $p$ - atom is a function $a(\cdot)$ supported on a cube $Q$ such that
\[
\| a \|_{\infty} \leq |Q|^{-1/p} \,\,\  \text{and} \,\, \int x^{\alpha} a(x) dx = 0 \,\,\, \text{for all} \,\, |\alpha| \leq 
n \left( \frac{1}{p}-1 \right).
\]
Such decompositions were obtained by R. Coifman in \cite{Coif} for the case $n=1$ and by R. Latter in \cite{Latter} for the case $n \geq 1$. 
These decompositions, in principle, allow to study the behavior of certain operators on $H^{p}(\mathbb{R}^{n})$ by focusing one's 
attention on individual atoms (see \cite{Bownik}, \cite{Ricci}, \cite{Pablo}). For instance, using such atomic decomposition, it obtains that the Riesz potential $\mathcal{R}_{\alpha}$ is bounded from $H^p(\mathbb{R}^n)$ into $L^q(\mathbb{R}^n)$, for $0 < p \leq 1$ and 
$\frac{1}{q} = \frac{1}{p} - \frac{\alpha}{n}$. In \cite{Taible}, M. H. Taibleson and G. Weiss, by means of a molecular decomposition for members in $H^p(\mathbb{R}^n)$, proved the $H^p(\mathbb{R}^n) \to H^q(\mathbb{R}^n)$ boundedness for $\mathcal{R}_{\alpha}$ on the full range 
$0 < p \leq 1$ (see also \cite{Krantz}). This extends the result obtained by E. Stein and G. Weiss above mentioned. A study about the 
behavior of Riesz potential $\mathcal{R}_{\alpha}$ on others function spaces was made by E. Nakai in \cite{Nakai}.  
 
For $0 < \alpha < n$, the discrete counterpart of (\ref{Ralfa}) is the discrete Riesz potential $I_{\alpha}$ on $\mathbb{Z}^n$, which is 
defined by
\begin{equation} \label{Riesz potential}
(I_{\alpha}b)(j) = \sum_{i \in \mathbb{Z}^n \setminus \{ j \}} \frac{b(i)}{|i-j |^{n - \alpha}}, \,\,\,\,\,\, j \in \mathbb{Z}^n.
\end{equation}
As one would expect, (\ref{Riesz potential}) has a similar behavior to that of (\ref{Ralfa}). Indeed, in \cite[p. 288]{Hardy}, for $n=1$, 
$0 < \alpha < 1$, $1 < p < 1/\alpha$ and $\frac{1}{q} = \frac{1}{p} - \alpha$, G. H. Hardy et al established the $\ell^p(\mathbb{Z}) - \ell^q(\mathbb{Z})$ boundedness of $I_{\alpha}$ (for the case $n \geq 1$ see \cite[Proposition $(a)$]{Wainger}, other proof was given in \cite[Theorem 3.1]{PRocha}).

The theory for Hardy spaces on $\mathbb{Z}^n$ was developed by S. Boza and M. Carro in \cite{Carro} (see also \cite{Boza}). There, the authors
gave a variety of distinct approaches to characterize the discrete Hardy spaces $H^p(\mathbb{Z}^n)$ analogous to the ones given for the Hardy spaces $H^p(\mathbb{R}^n)$. They also described an atomic decomposition for elements in $H^p(\mathbb{Z}^n)$.

Y. Kanjin and M. Satake in \cite{Kanjin}, based on some results of \cite{Boza}, constructed a molecular decomposition for 
$H^p(\mathbb{Z})$ analogous to the ones given by M. Taibleson and G. Weiss in \cite{Taible} for the Hardy spaces $H^p(\mathbb{R}^n)$. With this framework, they obtain the Marcinkiewicz multiplier theorem on $H^p(\mathbb{Z})$ and proved the $H^p(\mathbb{Z}) \to H^q(\mathbb{Z})$ boundedness of discrete Riesz potential $I_{\alpha}$, for $n=1$, $0 < p < \alpha^{-1}$ and $\frac{1}{q} = \frac{1}{p} - \alpha$. The 
$H^{p}(\mathbb{Z}) \to \ell^{q}(\mathbb{Z})$ boundedness of (\ref{Riesz potential}), with $n=1$, was observed by the author in 
\cite[Remark 11]{Rocha}. 

In \cite{PRocha}, by means of the atomic characterization of $H^p(\mathbb{Z}^n)$ developed in \cite{Carro} and the boundedness of discrete fractional maximal operator, we proved that $I_{\alpha}$ is a bounded operator $H^{p}(\mathbb{Z}^n) \to \ell^{q}(\mathbb{Z}^n)$ for $n \geq 1$,
$0 < p < \frac{n}{\alpha}$ and $\frac{1}{q} = \frac{1}{p} - \frac{\alpha}{n}$. 

The $H^p(\mathbb{Z}^n) \to H^q(\mathbb{Z}^n)$ boundedness of $I_{\alpha}$ for $\frac{n-1}{n} < p \leq q \leq 1$, was proved by the author in \cite{Rocha2}. To achieve this result we used the characterization of $H^p(\mathbb{Z}^n)$ by the discrete Riesz transforms and furnished 
a molecular decomposition, as in \cite{Kanjin}, for the elements of $H^{p}(\mathbb{Z}^n)$ on the range $\frac{n-1}{n} < p \leq 1$ with 
$n \geq 1$. We point out that the lower bound $\frac{n-1}{n}$ in this decomposition is a consequence of Theorem 2.6 in \cite{Carro}.

The purpose of this note is to extend the $H^p(\mathbb{Z}^n) \to H^q(\mathbb{Z}^n)$ boundedness of $I_{\alpha}$ obtained in \cite{Rocha2} to the full range $0 < p \leq 1$.

The main result of this note is contained in the following theorem, which will be proved in Section 3 by means of the characterization 
maximal for $H^p(\mathbb{Z}^n)$ established in \cite[Theorem 2.7]{Carro} and the atomic decomposition for $H^p(\mathbb{Z}^n)$ also given in \cite{Carro}, joint with some auxiliary results of Section 2.

\

{\sc Theorem} \ref{main result}.
{\it For $0 < \alpha < n$, let $I_{\alpha}$ be the discrete Riesz potential given by (\ref{Riesz potential}). If
$0 < p \leq 1$ and $\frac{1}{q} = \frac{1}{p} - \frac{\alpha}{n}$, then
\[
\| I_{\alpha} \, b \|_{H^{q}(\mathbb{Z}^n)} \leq C \| b \|_{H^{p}(\mathbb{Z}^n)}
\]
for all $b \in H^{p}(\mathbb{Z}^n)$, where $C$ does not depend on $b$.}

\

{\bf Notation.} Throughout this paper, $C$ will denote a positive real constant not necessarily the same at each occurrence. We set 
$\mathbb{Z}_{+} = \mathbb{N} \cup \{0\}$. For every $A \subset \mathbb{Z}^n$, we denote by $\#A$ and $\chi_{A}$ the cardinality of 
the set $A$ and the characteristic sequence of $A$ on $\mathbb{Z}^n$ respectively. Given a real number $s \geq 0$, we write 
$\lfloor s \rfloor$ for the integer part of $s$. For $j=(j_1, ..., j_n) \in \mathbb{Z}^n$ we consider the two following norms 
$|j| = (j_1^2 + \cdot \cdot \cdot + j_n^2)^{1/2}$ and $|j|_{\infty} = \max \{ |j_k| : k = 1, ..., n \}$. Given $k^0 \in \mathbb{Z}^n$ and 
$m \in \mathbb{Z}_{+}$ fix, we put $Q_{k^0, m} = \{ i \in \mathbb{Z}^n : |i - k^0|_{\infty} \leq m \}$, that is the discrete cube centered 
at $k^0$ and side length $2m + 1$. If $\beta$ is the multiindex $\beta =(\beta_{1},...,\beta _{n}) \in \mathbb{Z}_{+}^{n}$, then 
$[\beta] :=\beta _{1}+...+\beta _{n}$ and $j^{\beta} := j_{1}^{\beta_1} \cdot \cdot \cdot j_{n}^{\beta_n}$ for every $j=(j_1, ..., j_n) \in \mathbb{Z}^n$. By $B_r({\bf 0 })$ we denote the Euclidean ball in $\mathbb{R}^n$ with center at ${\bf 0 }$ and radius $r > 0$. Given a function $F$ defined on $\mathbb{R}^n$, we shall use the notation $F^d$ to indicate its restriction on $\mathbb{Z}^n$. 

\section{Preliminaries}

The Schwartz space $\mathcal{S}(\mathbb{R}^{n})$ is defined by
\[
\mathcal{S}(\mathbb{R}^{n}) = \left\{ \phi \in C^{\infty}(\mathbb{R}^{n}) : \sup_{x \in \mathbb{R}^{n}} (1+|x|)^{N} 
|(\partial^{\beta}\phi)(x)| < \infty \,\,\, \forall \,\, N \in \mathbb{Z}_{+}, \, \beta \in \mathbb{Z}_{+}^{n}  \right\}.
\]
We topologize the space $\mathcal{S}(\mathbb{R}^{n})$ with the following family of seminorms
\[
\| \phi \|_{\mathcal{S}(\mathbb{R}^{n}), \, N} = \sum_{[\beta] \leq N} \sup_{x \in \mathbb{R}^{n}} (1+|x|)^{N} |(\partial^{\beta} \phi)(x)|, \,\,\,\,\,\,\, (N \in \mathbb{Z}_{+}).
\]

For $0 < p < \infty$ and a sequence $b = \{ b(i) \}_{i \in \mathbb{Z}^n}$ we say that $b$ belongs to $\ell^{p}(\mathbb{Z}^n)$ if
\[
\| b \|_{\ell^p(\mathbb{Z}^n)} :=\left( \sum_{i \in \mathbb{Z}^n} |b(i)|^p \right)^{1/p} < \infty.
\]
For $p=\infty$, we say that $b$ belongs to $\ell^{\infty}(\mathbb{Z}^n)$ if 
\[
\|b \|_{\ell^\infty(\mathbb{Z}^n)} := \sup_{i \in \mathbb{Z}^n} |b(i)| < \infty.
\]

Given two sequences $b=\{ b(i) \}_{i \in \mathbb{Z}^n}$ and $c=\{ c(i) \}_{i \in \mathbb{Z}^n}$, their convolution $b \ast c$ is defined by
\[
(b \ast c)(j) = \sum_{i \in \mathbb{Z}^n} b(i) \, c(j-i),
\]
when the series converges (e.g. \cite[Section 1.2.3, p. 21-25]{Loukas}). Moreover, $b \ast c = c \ast b$.

We recall that a discrete cube $Q$ centered at $j =(j_1, ..., j_n)  \in \mathbb{Z}^n$ can be written of the form 
$Q = Q_{j, m} = \prod_{1 \leq l \leq n} [[j_l -m, j_l +m]]$, where for each $l=1, ..., n$, $[[j_l -m, j_l +m]] := [j_l -m, j_l +m] \cap \mathbb{Z} = \{ j_l - m, ..., j_l, ..., j_l + m \}$ with $m \in \mathbb{Z}_{+}$. It is clear that $\# Q = (2m+1)^n$.

Let $0 \leq \alpha < n$, given a sequence $b = \{ b(i) \}_{i \in \mathbb{Z}^n}$ we define the centered fractional maximal sequence 
$M_{\alpha} b$ by
\[
(M_{\alpha} b)(j) = \sup_{Q \ni j} \frac{1}{(\# Q)^{1 - \frac{\alpha}{n}}} \sum_{i \in Q} |b(i)|, \,\,\,\,\,\, j \in \mathbb{Z}^n,
\]
where the supremum is taken over all discrete cubes $Q$ centered at $j$. We observe that if $\alpha = 0$, then $M_0 = M$ where $M$ is the centered discrete maximal operator of Hardy-Littlewood. For $0 \leq \alpha < n$, by Theorem 2.3 and Proposition 2.4 in \cite{PRocha}, we have that the operator $M_{\alpha}$ is bounded from $\ell^{p}(\mathbb{Z}^n)$ into $\ell^{q}(\mathbb{Z}^n)$ for $1 < p < \frac{n}{\alpha}$ and 
$\frac{1}{q} = \frac{1}{p} - \frac{\alpha}{n}$.

We adopt the following definition of discrete Hardy space on $\mathbb{Z}^n$ (see \cite[Theorem 2.7]{Carro}). Let $\Phi \in \mathcal{S}(\mathbb{R}^n)$ with $\int_{\mathbb{R}^n} \Phi =1$. Then, for $t>0$, we put $\Phi_t^d(j) = t^{-n} \Phi(j/t)$ if $j \in \mathbb{Z}^n \setminus \{ {\bf 0} \}$ and $\Phi_t^d({\bf 0}) = 0$. Now, for $0 < p \leq \infty$, we define
\begin{equation} \label{discrete Hp}
H^p(\mathbb{Z}^n) := \left\{ b \in \ell^p(\mathbb{Z}^n) : \sup_{t>0} |(\Phi_t^d \ast b)| \in \ell^p(\mathbb{Z}^n) \right\},
\end{equation}
equipped with the quasi-norm given by
\begin{equation} \label{discrete norm}
\| b \|_{H^p(\mathbb{Z}^n)} := \| b \|_{\ell^p(\mathbb{Z}^n)} + \|\sup_{t>0} |(\Phi_t^d \ast b)| \|_{\ell^p(\mathbb{Z}^n)}.
\end{equation}

In \cite{Carro}, Boza and Carro also gave an atomic characterization of $H^p(\mathbb{Z}^n)$ for $0 < p \leq 1$. Before establishing this result we recall the definition of $(p, \infty, N_p)$-atom in $H^p(\mathbb{Z}^n)$. 

\begin{definition} Let $0 < p \leq 1$ and $N_p := \lfloor n(p^{-1} - 1) \rfloor$. We say that a sequence 
$a = \{ a(j) \}_{j \in \mathbb{Z}^n}$ is an $(p, \infty, N_p)$-atom centered at a discrete cube $Q \subset \mathbb{Z}^n$ if the following three conditions hold:

(a1) $\supp a \subset Q$,

(a2.$p$) $\| a \|_{\ell^\infty(\mathbb{Z}^n)} \leq (\# Q)^{-1/p}$,

(a3.$p$) $\displaystyle{\sum_{j \in Q}} j^{\beta} a(j) = 0$ for every multi-index $\beta=(\beta_1, ..., \beta_n) \in \mathbb{Z}_{+}^{n}$ with 
$[\beta] \leq N_p$.
\end{definition}

The atomic decomposition mentioned for $H^p(\mathbb{Z}^n)$ is established in the following theorem.

\begin{theorem} (\cite[Theorem 3.7]{Carro}) \label{atomic Hp} Let $0 < p \leq 1$, $N_p = \lfloor n (p^{-1} - 1) \rfloor$ and 
$b \in H^{p}(\mathbb{Z}^n)$. Then there exist a sequence of $(p, \infty, N_p)$-atoms $\{ a_k \}_{k=0}^{+\infty}$, a sequence of scalars 
$\{ \lambda_k \}_{k=0}^{+\infty}$ and a positive constant $C$, which depends only on $p$ and $n$, with 
$\sum_{k=0}^{+\infty} |\lambda_k |^{p} \leq C \| b \|_{H^{p}(\mathbb{Z}^n)}^{p}$ such that $b = \sum_{k=0}^{+\infty} \lambda_k a_k$, where the series converges in $H^{p}(\mathbb{Z}^n)$.
\end{theorem}

\begin{remark} \label{Hp equiv}
From \cite[Corollary 4]{Rocha}, it follows for every $0 < p \leq 1$ that $H^p(\mathbb{Z}^n) \subset \ell^p(\mathbb{Z}^n)$ strictly. If 
$\Phi \in \mathcal{S}(\mathbb{R}^n)$ is such that $\supp(\Phi) \subset B_{1/4}({\bf 0 })$, then 
$\sup_{t>0} |(\Phi_t^d \ast b)(j)| \leq C (Mb)(j)$ for all $j \in \mathbb{Z}^n$, where $M$ is the discrete maximal operator of 
Hardy-Littlewood. Thus, by \cite[Theorem 2.3]{PRocha}, $H^p(\mathbb{Z}^n) = \ell^p(\mathbb{Z}^n)$ for $1 < p \leq \infty$, with equivalent norms. 
\end{remark}

We conclude these preliminaries with the following two auxiliary results.

\begin{lemma} For $0 < \alpha < n$, $m \in \mathbb{N}$ and $k^0, k \in \mathbb{Z}^n$ fix, the estimate
\begin{equation} \label{kernel estimate}
\sum_{i \in Q_{k^0, m} \setminus \{ k \}} |k-i|^{\alpha - n} \leq C (2m+1)^{\alpha}
\end{equation}
holds with $C$ independent of $m \in \mathbb{N}$ and $k^0, k \in \mathbb{Z}^n$.
\end{lemma}

\begin{proof}
Since $|j|_{\infty} \leq |j|$ for all $j \in \mathbb{Z}^n$, we have
\[
\sum_{i \in Q_{k^0, m} \setminus \{ k \}} |k-i|^{\alpha - n} \leq \sum_{i \in Q_{k^0, m} \setminus \{ k \}} |k-i|^{\alpha - n}_{\infty}
= \sum_{i \in I_{k^0, k, m}} |k-i|^{\alpha - n}_{\infty} + \sum_{i \in J_{k^0, k, m}} |k-i|^{\alpha - n}_{\infty},
\]
where
\[
I_{k^0, k, m} = \{ i \in Q_{k^0, m} : |k-i|_{\infty} > m \} \,\,\, \text{and} \,\,\, J_{k^0, k, m} = \{ i \in Q_{k^0, m} : 0<|k-i|_{\infty} 
\leq m \}.
\]
Now, one has
\[
\sum_{i \in I_{k^0, k, m}} |k-i|^{\alpha - n}_{\infty} \leq m^{\alpha - n} \cdot \#Q_{k^0, m} \leq C (2m+1)^{\alpha}.
\]
On the other hand, there exists a unique $s_0 \in \mathbb{N}$ such that $2^{s_0 - 1} \leq m < 2^{s_0}$. Then,
\[
\sum_{i \in J_{k^0, k, m}} |k-i|^{\alpha - n}_{\infty} \leq \sum_{s=0}^{s_0 - 1} \sum_{2^{-(s+1)} m < |k-i|_{\infty} \leq 2^{-s} m} 
|k-i|^{\alpha - n}_{\infty}
\]
\[
\leq m^{\alpha} \sum_{s=0}^{s_0 - 1} 2^{-(s+1) \alpha} \, \left( (2^{-(s+1)} m)^{-n} \sum_{i: |k-i|_{\infty} \leq 2^{-s} m} 1 \right)
\]
\[
\leq C (2m+1)^{\alpha} \sum_{s=0}^{\infty} 2^{-(s+1) \alpha}.
\]
Thus, (\ref{kernel estimate}) follows.
\end{proof}

\begin{lemma} \label{phimax}
Let $0 < \alpha < n$ and $\Phi \in \mathcal{S}(\mathbb{R}^n)$ such that $\supp(\Phi) \subset B_{1/4}({\bf 0})$. If $I_{\alpha}$ is the discrete Riesz potential given by (\ref{Riesz potential}), then there exists a positive constant $C = C_{\Phi, n, \alpha}$ such that
\begin{equation} \label{ineq puntual}
\sup_{t>0} |(\Phi_t^d \ast I_{\alpha}a)(j)| \leq C (\# Q)^{\frac{\alpha}{n} - \frac{1}{p}},
\end{equation}
holds for all $j \in \mathbb{Z}^n$ and all $p$-discrete atoms $a(\cdot)$ supported on $Q$.
\end{lemma}

\begin{proof}
Let $a(\cdot)$ be a $p$-atom supported by the cube $Q_{k^0, m} = \{ i : | i - k^0 |_{\infty} \leq m \}$ centered at $k^0$ and 
side length $2m+1$. Since by hypothesis $\supp(\Phi) \subset B_{1/4}({\bf 0})$ and $\Phi^{d}_{t}({\bf 0}) = 0$, we have
\[
(\Phi_t^d \ast I_{\alpha}a)(j) = 0, 
\]
for all $0 < t < 1$ and $j \in \mathbb{Z}^n$.

For $t \geq 1$ we have $1 \leq \lfloor t \rfloor \leq t$, then the condition (a2.$p$) of the atom $a(\cdot)$ and the estimate 
(\ref{kernel estimate}) give
\begin{align*}
|(\Phi_t^d \ast I_{\alpha}a)(j)| &\leq \sum_{k} |\Phi_{t}(j-k)| \sum_{i \in Q_{k^0, m} \setminus \{ k \}} |a(i)| |k-i|^{\alpha - n} \\
&\leq C (\# Q)^{\frac{\alpha}{n} - \frac{1}{p}} \left( t^{-n} \sum_{|k|_{\infty} \leq \frac{t}{4}} |\Phi(k/t)| \right) \\
&\leq C (\# Q)^{\frac{\alpha}{n} - \frac{1}{p}} \left( \lfloor t \rfloor^{-n} \sum_{|k|_{\infty} \leq \lfloor t \rfloor} 
|\Phi(k/t)| \right) \\
&\leq C \|\Phi \|_{\infty} (\# Q)^{\frac{\alpha}{n} - \frac{1}{p}},
\end{align*}
for all $1 \leq t < \infty$ and $j \in \mathbb{Z}^n$. Thus, (\ref{ineq puntual}) follows.
\end{proof}

\section{Main result}

In this section we establish the $H^p(\mathbb{Z}^n) - H^q(\mathbb{Z}^n)$ boundedness of the discrete Riesz potential $I_{\alpha}$ on the full range $0 < p \leq 1$, where $\frac{1}{q} = \frac{1}{p} - \frac{\alpha}{n}$. For them, we first show that the operator $I_{\alpha}$ is bounded uniformly in the $H^q(\mathbb{Z}^n)$-norm on all $(p, \infty, N_p)$-atoms.

\begin{proposition} \label{maximal atom estim}
Let $0 < \alpha < n$ and $\Phi \in \mathcal{S}(\mathbb{R}^n)$ such that $\supp(\Phi) \subset B_{1/4}({\bf 0})$. If $I_{\alpha}$ is the discrete Riesz potential given by (\ref{Riesz potential}), then, for $0 < p \leq 1$ and $\frac{1}{q} = \frac{1}{p} - \frac{\alpha}{n}$, there exists a positive constant $C$ such that
\begin{equation} \label{uniform estimate}
\|\sup_{t>0}|(\Phi_{t}^d \ast I_{\alpha} a)| \|_{\ell^{q}(\mathbb{Z}^n)} \leq C,
\end{equation}
for all discrete $(p, \infty, N_p)$-atom $a = \{ a(i) \}$.
\end{proposition}

\begin{proof}
To prove (\ref{uniform estimate}), let $a(\cdot)$ be a discrete $(p, \infty, N_p)$-atom centered at the cube 
$Q_{k^0}= \displaystyle{\prod_{1 \leq l \leq n}}\left[\left[ k^0_l - m, k^0_l + m \right]\right]$. We put
$4\lfloor \sqrt{n} \rfloor Q_{k^0} = \displaystyle{\prod_{1 \leq l \leq n}} \left[\left[ k^0_l - 4\lfloor \sqrt{n} \rfloor m, k^0_l + 4\lfloor \sqrt{n} \rfloor m \right]\right]$. Then, we split
\begin{align}
\sum_{j \in \mathbb{Z}^n} |\sup_{t>0}|(\Phi_{t}^d \ast I_{\alpha} \, a)(j)|^{q} &= \sum_{j \in 4\lfloor \sqrt{n} \rfloor Q_{k^0}} 
|\sup_{t>0}|(\Phi_{t}^d \ast I_{\alpha} \, a)(j)|^{q} \label{sum2} \\
&+ \sum_{j \in \mathbb{Z}^n \setminus 4\lfloor \sqrt{n} \rfloor Q_{k^0}} |\sup_{t>0}|(\Phi_{t}^d \ast I_{\alpha} \, a)(j)|^{q}. \nonumber
\end{align} 
To estimate the first sum in (\ref{sum2}), we apply Lemma \ref{phimax} above, which leads to
\begin{align}
\sum_{j \in 4\lfloor \sqrt{n} \rfloor Q_{k^0}}|\sup_{t>0}|(\Phi_{t}^d \ast I_{\alpha} \, a)(j)|^{q} &\leq  
C (\# Q_{k^0}) (\# Q_{k^0})^{q(\frac{\alpha}{n} - \frac{1}{p})} = C, \label{estim C}
\end{align}
with $C$ independent of $k^0$ and $m$.

To estimate the second sum in (\ref{sum2}), we analize the cases $0 < t < 1$ and $t \geq 1$ separately. For $0 < t < 1$, by taking into account that $\supp(\Phi) \subset B_{1/4}({\bf 0})$ and $\Phi^{d}_{t}({\bf 0}) = 0$, we have 
\begin{equation} \label{cero}
(\Phi_t^d \ast I_{\alpha}a)(j) = 0
\end{equation}
for all $j \in \mathbb{Z}^n$. For $t \geq 1$ we consider two sub-cases,

\

{\bf Case I:} $|j - k^0 |_{\infty} > 4 \lfloor \sqrt{n} \rfloor m$ and $1 \leq t \leq 2 |j - k^0 |_{\infty}$. We have
\[
(\Phi_t^d \ast I_{\alpha}a)(j) = \sum_{k : |j-k|_{\infty} \leq \frac{t}{4}} \Phi^{d}_{t}(j-k) \sum_{i \in Q_{k^0}} a(i) |i-k|^{\alpha-n}.
\]
Now, $|j-k|_{\infty} \leq t/4 \leq |j-k^0|_{\infty}/2$ and $|j - k^0 |_{\infty} > 4 \lfloor \sqrt{n} \rfloor m$ imply that 
$|k - k^0 |_{\infty} \geq |j - k^0 |_{\infty}/2 \geq 2|i - k^0 |_{\infty}$. We put $N - 1 = \lfloor n(p^{-1} - 1) \rfloor$. In view of the moment condition (a3.$p$) of $a(\cdot)$ we have, for $j \in \mathbb{Z}^n \setminus 4 \lfloor \sqrt{n} \rfloor Q_{k^0}$, that
\[
(\Phi_{t}^{d} \ast I_{\alpha} \, a)(j) = \sum_{k : |j-k|_{\infty} \leq \frac{t}{4}} \Phi^{d}_{t}(j-k) \sum_{i \in Q_{k^0}} a(i) 
\left[|i-k|^{\alpha-n} - q_{N}(i,k) \right],
\]
where $q_{N}(\, \cdot \,, k)$ is the degree $N - 1$ Taylor polynomial of the function $x \rightarrow |k - x|^{\alpha-n}$ expanded around 
$k^0$. By the standard estimate of the remainder term in the Taylor expansion there exists $\xi$ between $i$ and $k^0$ such that
\[
\left| |i-k|^{\alpha-n} - q_{N}(i,k) \right| \leq C |i - k^0 |^{N} |k - \xi|^{\alpha-n-N}.
\]
Since $|k - k^0 |_{\infty} \geq 2|i - k^0 |_{\infty}$ for any $i \in Q_{k^0}$, we get 
$|k - \xi| \geq |k - \xi|_{\infty} \geq |k - k^0|_{\infty}/2$. So,
\[
\left| |i-k|^{\alpha-n} - q_{N}(i,k) \right| \leq C |i - k^0 |^{N}_{\infty} |k - k^0|^{\alpha-n-N}_{\infty}.
\]
As also $|k - k^0 |_{\infty} \geq |j - k^0 |_{\infty}/2$ for any $j \notin 4\lfloor \sqrt{n} \rfloor Q_{k^0}$, it follows
\[
\left| |i-k|^{\alpha-n} - q_{N}(i,k) \right| \leq C |i - k^0 |^{N}_{\infty} |j - k^0|^{\alpha-n-N}_{\infty},
\]
with $C$ independent of $i$, $j$, $k$ and $k^0$. This inequality, the condition (a2.$p$) of the atom $a(\cdot)$ and the fact that 
$\{ k : |k|_{\infty} \leq t/4 \} \subset \{ k : |k|_{\infty} \leq \lfloor t \rfloor \}$, where $1 \leq \lfloor t \rfloor \leq t$, allow us to conclude that
\[
|(\Phi^{d}_{t} \ast I_{\alpha}a)(j)| \leq C \frac{(2m+1)^{n+N}}{(\# Q_{k^0})^{1/p}} | j - k^0|^{\alpha-n-N}_{\infty} 
\left( \lfloor t \rfloor^{-n} \sum_{|k|_{\infty} \leq \lfloor t \rfloor} |\Phi(k/t)| \right)
\]
\[
\leq C \| \Phi \|_{\infty} \frac{(2m+1)^{n+N}}{(\# Q_{k^0})^{1/p}} | j - k^0|^{\alpha-n-N}_{\infty}
\]
\[
\leq \frac{C}{(\# Q_{k^0})^{1/p}} \left( \frac{(2m+1)^n}{(| j - k^0|_{\infty}^{n})^{1-\frac{\alpha}{n+N}}} \right)^{\frac{n+N}{n}}
\]
\[
\leq \frac{C}{(\# Q_{k^0})^{1/p}} \left( \frac{1}{(\# Q_{j, 2| j - k^0|_{\infty}})^{1-\frac{\alpha}{n+N}}} 
\sum_{i \in Q_{j, 2| j - k^0|_{\infty}}} \chi_{Q_{k^0}}(i) \right)^{\frac{n+N}{n}}
\]
where $Q_{j, 2| j - k^0|_{\infty}}$ is the discrete cube centered at $j$ and side length $4| j - k^0|_{\infty} + 1$. Thus
\[
|(\Phi^{d}_{t} \ast I_{\alpha}a)(j)| \leq \frac{C}{(\# Q_{k^0})^{1/p}} 
\left[ M_{\frac{\alpha n}{n+N}} (\chi_{Q_{k^0}})(j) \right]^{\frac{n+N}{n}}
\]
for all $j \notin 4 \lfloor \sqrt{n} \rfloor Q_{k^0}$ and $1 \leq t \leq 2 |j - k^0 |_{\infty}$. 

\

{\bf Case II:} $|j - k^0 |_{\infty} > 4 \lfloor \sqrt{n} \rfloor m$ and $t > 2 |j - k^0 |_{\infty}$. We have that
\[
(\Phi_t^d \ast I_{\alpha}a)(j) = \sum_{l : |j-l|_{\infty} \leq \frac{t}{4}} \Phi^{d}_{t}(j-l) \sum_{i \in Q_{k^0}} a(i) \, |i-l|^{\alpha-n}
\]
\[
= \sum_{i \in Q_{k^0}} a(i) \sum_{l : |j-l|_{\infty} \leq \frac{t}{4}} \Phi^{d}_{t}(j-l) \, |i-l|^{\alpha-n}
\]
\[
= \sum_{i \in Q_{k^0}} a(i) \sum_{k} \Phi^{d}_{t}(j- k -i) \, |k|^{\alpha-n},
\]
where $k = l - i$, in this case $|j-k-i|_{\infty} \leq t/4$ implies that the inner sum is vanishing on $|k|_{\infty} > 2t$, so
\[
(\Phi_t^d \ast I_{\alpha}a)(j) = \sum_{0 < |k|_{\infty} \leq 2t} |k|^{\alpha-n} \sum_{i \in Q_{k^0}} a(i) \, \Phi^{d}_{t}(j- k -i).
\]
As before, we put $N - 1 = \lfloor n(p^{-1} - 1) \rfloor$. Then, for $j \in \mathbb{Z}^n \setminus 4 \lfloor \sqrt{n} \rfloor Q_{k^0}$, the moment condition (a3.$p$) of $a(\cdot)$ gives
\[
(\Phi_t^d \ast I_{\alpha}a)(j) = \sum_{0< |k|_{\infty} \leq 2t} |k|^{\alpha-n} \sum_{i \in Q_{k^0}} a(i) \, 
\left[\Phi^{d}_{t}(j- k -i) - \widetilde{q}_N(i, j, k, t) \right],
\]
where $\widetilde{q}_N(\, \cdot \,, j, k, t)$ is the degree $N - 1$ Taylor polynomial of the function $x \rightarrow \Phi_{t}(j- k -x)$ expanded around $k^0$. By the standard estimate of the remainder term in the Taylor expansion we have that
\[
\left| \Phi^{d}_{t}(j- k -i) - \widetilde{q}_N(i, j, k, t) \right| \leq C \| \Phi \|_{\mathcal{S}(\mathbb{R}^n), N} \, |i - k^0|^{N} t^{-n-N},
\]
where $C$ does not depend on $i$, $j$, $k$ and $t$. This inequality and the condition (a2.$p$) of the atom $a(\cdot)$ lead to
\[
\left|(\Phi_t^d \ast I_{\alpha}a)(j)\right| \leq C \frac{(2m+1)^{n+N}}{(\# Q_{k^0})^{1/p}}t^{-n-N}\sum_{0 < |k|_{\infty} \leq 2t} |k|^{\alpha-n}.
\]
Now, the estimate (\ref{kernel estimate}) and $t > 2 |j - k^0 |_{\infty}$ imply
\[
\left|(\Phi_t^d \ast I_{\alpha}a)(j)\right| \leq C \frac{(2m+1)^{n+N}}{(\# Q_{k^0})^{1/p}} |j - k^0 |_{\infty}^{\alpha-n-N},
\]
\[
\leq \frac{C}{(\# Q_{k^0})^{1/p}} 
\left[ M_{\frac{\alpha n}{n+N}} (\chi_{Q_{k^0}})(j) \right]^{\frac{n+N}{n}},
\]
for all $j \notin 4 \lfloor \sqrt{n} \rfloor Q_{k^0}$ and $t > 2 |j - k^0 |_{\infty}$. 

Thus, (\ref{cero}) and the estimates obtained in the cases $I$ and $II$ above give
\begin{equation} \label{cota afuera}
\sum_{j \in \mathbb{Z}^n \setminus 4\lfloor \sqrt{n} \rfloor Q_{k^0}} \sup_{t>0}|(\Phi^{d}_{t} \ast I_{\alpha}a)(j)|^q \leq 
\frac{C}{(\# Q_{k^0})^{q/p}} \sum_{j \in \mathbb{Z}^n} \left[ M_{\frac{\alpha n}{n+N}} (\chi_{Q_{k^0}})(j) \right]^{q \frac{n+N}{n}}.
\end{equation}
Since  $N-1= \lfloor n(\frac{1}{p}-1) \rfloor$, we have $q \frac{n+N}{n} > 1$. We write $\widetilde{q} = q \frac{n+N}{n}$ and let 
$\frac{1}{\widetilde{p}} = \frac{1}{\widetilde{q}} + \frac{\alpha}{n+N}$, so $\frac{\widetilde{p}}{\widetilde{q}} = \frac{p}{q}$. From
Proposition 2.4 in \cite{PRocha}, we obtain
\begin{equation} \label{cota afuera 2}
\sum_{j \in \mathbb{Z}^n} \left[ M_{\frac{\alpha n}{n+N}} (\chi_{Q_{k^0}})(j) \right]^{q \frac{n+N}{n}} \leq
C \left( \sum_{j \in \mathbb{Z}^n} \chi_{Q_{k^0}}(j) \right)^{q/p} = C (\# Q_{k^0})^{q/p}.
\end{equation}
Now, (\ref{cota afuera}) and (\ref{cota afuera 2}) give
\begin{equation} \label{cota afuera 3}
\sum_{j \in \mathbb{Z}^n \setminus 4\lfloor \sqrt{n} \rfloor Q_{k^0}} \sup_{t>0}|(\Phi^{d}_{t} \ast I_{\alpha}a)(j)|^q \leq C.
\end{equation}
Finally, (\ref{estim C}) and (\ref{cota afuera 3}) lead to (\ref{uniform estimate}). Thus the proof is finished.
\end{proof}

\begin{theorem} \label{main result}
For $0 < \alpha < n$, let $I_{\alpha}$ be the discrete Riesz potential given by (\ref{Riesz potential}). If
$0 < p \leq 1$ and $\frac{1}{q} = \frac{1}{p} - \frac{\alpha}{n}$, then
\[
\| I_{\alpha} \, b \|_{H^{q}(\mathbb{Z}^n)} \leq C \| b \|_{H^{p}(\mathbb{Z}^n)}
\]
for all $b \in H^{p}(\mathbb{Z}^n)$, where $C$ does not depend on $b$.
\end{theorem}

\begin{proof}
We fix $p_0$ such that $1 < p_0 < \frac{n}{\alpha}$ and $\Phi \in \mathcal{S}(\mathbb{R}^n)$ with $\supp(\Phi) \subset B_{1/4}({\bf 0 })$ 
and $\int \Phi = 1$. By Theorem \ref{atomic Hp}, given $b \in H^{p}(\mathbb{Z}^n)$, with $0 < p \leq 1$, we can write 
$b = \sum_k \lambda_k a_k$ where the $a_k$'s are $(p, \infty, N_p)$ atoms, the scalars $\lambda_k$ satisfies 
$\sum_{k} |\lambda_k |^{p} \leq C \| b \|_{H^{p}(\mathbb{Z}^n)}^{p}$  and the series converges in $H^{p}(\mathbb{Z}^n)$ and 
so in $\ell^{p_0}(\mathbb{Z}^n)$ since $H^{p}(\mathbb{Z}^n) \subset \ell^{p}(\mathbb{Z}^n) \subset \ell^{p_0}(\mathbb{Z}^n)$ embed continuously. Now for $\frac{1}{q_0} = \frac{1}{p_0} - \frac{\alpha}{n}$, by \cite[Theorem 3.1]{PRocha}, the discrete Riesz potential 
$I_{\alpha}$ is a bounded operator $\ell^{p_0}(\mathbb{Z}^n) \to \ell^{q_0}(\mathbb{Z}^n)$ and since 
$H^{q_0}(\mathbb{Z}^{n}) = \ell^{q_0}(\mathbb{Z}^{n})$ with equivalent norms (see Remark \ref{Hp equiv}), it follows that
\begin{equation} \label{discrete maximal}
\sup_{t>0} |(\Phi_{t}^d \ast I_{\alpha}b)(j)| \leq \sum_{k=1}^{\infty} |\lambda_k| \sup_{t>0}|(\Phi_{t}^d \ast I_{\alpha} a_k)(j)|, \,\,\, \text{for all} \,\, j \in \mathbb{Z}^n.
\end{equation}
Then for $1 \leq q \leq \frac{n}{n - \alpha}$, by (\ref{discrete maximal}) and Minkowski's integral inequality on $\sigma$-finite measure spaces, we have
\begin{equation} \label{mink ineq}
\|\sup_{t>0} |(\Phi_{t}^d \ast I_{\alpha}b)| \|_{\ell^{q}(\mathbb{Z}^n)} \leq \sum_{k} |\lambda_k| \,
\| \sup_{t>0}|(\Phi_{t}^d \ast I_{\alpha} a_k)| \|_{\ell^{q}(\mathbb{Z}^n)}.
\end{equation}
Now for $0 < q < 1$, from (\ref{discrete maximal}), it is easy to check that
\begin{equation} \label{q ineq}
\| \sup_{t>0} |(\Phi_{t}^d \ast I_{\alpha}b)| \|_{\ell^{q}(\mathbb{Z}^n)}^q \leq \sum_{k} |\lambda_k|^q 
\| \sup_{t>0}|(\Phi_{t}^d \ast I_{\alpha} a_k)| \|_{\ell^{q}(\mathbb{Z}^n)}^q.
\end{equation}
Finally, since $0 < p \leq 1$ and $0 < p < q \leq \frac{n}{n - \alpha}$, the estimate (\ref{mink ineq}) or (\ref{q ineq}) according to the case, Proposition \ref{maximal atom estim}, and the fact that $\sum_{k} |\lambda_k |^{p} \leq C \| b \|_{H^{p}(\mathbb{Z}^n)}^{p}$  allow us to conclude that 
\[
\|\sup_{t>0}|(\Phi_{t}^d \ast I_{\alpha} b)| \|_{\ell^{q}(\mathbb{Z}^n)} \leq C \left( \sum_{k} |\lambda_k|^{\min\{1, q \}} \right)^{\frac{1}{\min\{1, q \}}} \leq C \left( \sum_{k} |\lambda_k |^{p} \right)^{1/p} \leq C\| b \|_{H^{p}(\mathbb{Z}^n)}.
\] 
Since $b$ is an arbitrary element of $H^{p}(\mathbb{Z}^n)$, the theorem follows.
\end{proof}

\begin{remark}
The argument used to prove Theorem 3.2 is similar to the one given in the proof of \cite[Theorem 2.4.5, see p. 140-141]{Grafakos}.
\end{remark}

{\bf Acknowledgements.} I am very grateful to the referee for the careful reading of the paper and for the useful comments and suggestions which helped me to improve the original manuscript.

Pablo Rocha, Instituto de Matem\'atica (INMABB), Departamento de Matem\'atica, Universidad Nacional del Sur (UNS)-CONICET, Bah\'ia Blanca, Argentina. \\
{\it e-mail:} pablo.rocha@uns.edu.ar

\end{document}